\documentclass[11pt]{article}
\usepackage[margin=1in]{geometry}
\usepackage{amsmath,amssymb,amsthm}
\usepackage[hidelinks]{hyperref}
\newtheorem{theorem}{Theorem}
\newtheorem{lemma}[theorem]{Lemma}
\newcommand{\NN}{\mathbb Z_{\geq0}}
\title{Rationality and Quasipolynomiality of Restricted Rectangle Partitions}
\author{Mubin Shaikh\\
\small Independent Researcher\\
\small\texttt{shaikhmubin572@gmail.com}\\
\small ORCID: \href{https://orcid.org/0009-0005-1290-7114}{https://orcid.org/0009-0005-1290-7114}}
\hypersetup{pdfauthor={Mubin Shaikh},
pdftitle={Rationality and Quasipolynomiality of Restricted Rectangle Partitions}}
\date{Preprint --- Version 1.0\\September 5, 2026}
\begin{document}
\maketitle
\begin{abstract}
This paper proves Conjecture 5.10 of Gajdzica, Visser, and Zakarczemny on
restricted partitions of a rectangle. For every fixed positive integer
$k$, the generating function for feasible multisets of bars of lengths
at most $k$ tiling a $2\times n$ rectangle has denominator dividing
$\prod_{j=1}^k(1-x^j)$. Its coefficients are eventually quasipolynomial
of degree $k-1$ and period dividing $\operatorname{lcm}(1,\ldots,k)$.
The key observation is that appending a two-bar slab makes feasibility
upward closed within each parity class of multiplicities. Dickson's
lemma and finite inclusion--exclusion then give the precise denominator.
\end{abstract}

\section{Statement and conventions}
Fix an integer $k\geq1$. A rectangle partition here is a \emph{multiset}
of bars of sizes $1\times j$, $1\leq j\leq k$, which can tile a
$2\times n$ rectangle. Rotations are allowed, and rotated copies are
the same block type; geometric arrangements of a given multiset are
not distinguished. In particular a $1\times2$ bar may be vertical.
Let $f_k(n)$ count these multisets, with $f_k(0)=1$.

In the journal version of Gajdzica--Visser--Zakarczemny
\cite[Section 5.1, Conjecture 5.10]{GVZ}, this function is denoted
$p_{k,1}(2,n)$. The same conjecture appears in arXiv v2 as
Conjecture 5.10, p.~17. In v1 it is Conjecture 3.10 with notation
$p_{k,0}(2,n)$ \cite{pre}. The following theorem establishes the full assertion:
\begin{theorem}\label{thm:main}
For every integer $k\geq1$, there is a polynomial $P_k(x)\in\mathbb Z[x]$
such that, as formal power series,
\begin{equation}\label{eq:gf}
 \sum_{n\geq0} f_k(n)x^n
 =\frac{P_k(x)}{\prod_{j=1}^k(1-x^j)}.
\end{equation}
Consequently the reduced denominator divides the displayed product.
The function $f_k(n)$ is eventually quasipolynomial of degree $k-1$,
with quasiperiod dividing $\operatorname{lcm}(1,\ldots,k)$.
\end{theorem}
A quasipolynomial of period $L$ means a function given by one polynomial
in $n$ on each residue class modulo $L$; its degree is the maximum of
their degrees. ``Eventually'' permits finitely many exceptional initial
values. Changing only the convention at $n=0$ would change the generating
function by a polynomial and would not affect the theorem.

\section{Finite generation of upper sets}
Equip $\NN^d$ with the coordinatewise order. The next elementary form
of Dickson's lemma is included to make the argument self-contained.
\begin{lemma}\label{lem:dickson}
Every subset of $\NN^d$ has finitely many coordinatewise minimal
elements. If $U\subseteq\NN^d$ is upward closed, then
\[
 U=\bigcup_{a\in M}(a+\NN^d),
\]
where $M$ is its finite set of minimal elements.
\end{lemma}
\begin{proof}
Every infinite sequence of nonnegative integers has an infinite
nondecreasing subsequence. Indeed, either one value occurs infinitely
often, or every bounded set contains only finitely many terms and a
strictly increasing subsequence can be selected successively.
Apply this observation successively to each coordinate of an infinite
sequence in $\NN^d$. The resulting infinite subsequence is nondecreasing
in every coordinate. Thus an infinite set of distinct mutually
incomparable elements is impossible. Distinct minimal elements are
incomparable, proving finiteness.

For $u\in U$, the nonempty finite set
$\{v\in U:v\leq u\}$ contains a minimal element $a$. Any element of
$U$ strictly below $a$ would also be below $u$, so $a$ is minimal in
$U$ itself. Hence $u\in a+\NN^d$. The reverse inclusion follows
from upward closure. The empty set is covered by taking $M=\varnothing$.
\end{proof}

For positive integer weights $w=(w_1,\ldots,w_d)$, put
$w\cdot a=\sum_jw_ja_j$.
\begin{lemma}\label{lem:series}
If $U\subseteq\NN^d$ is upward closed, then
\[
 \sum_{u\in U}x^{w\cdot u}
 =\frac{A_U(x)}{\prod_{j=1}^d(1-x^{w_j})}
 \quad\text{for some }A_U(x)\in\mathbb Z[x].
\]
More explicitly, for the finite minimal set $M$,
\begin{equation}\label{eq:ie}
 A_U(x)=\sum_{\varnothing\ne S\subseteq M}
 (-1)^{|S|+1}x^{w\cdot\bigvee S},
\end{equation}
where $\bigvee S$ is the coordinatewise maximum of the vectors in $S$.
\end{lemma}
\begin{proof}
For each nonempty $S\subseteq M$,
\[
 \bigcap_{a\in S}(a+\NN^d)=(\bigvee S)+\NN^d,
 \qquad
 \sum_{u\in(\bigvee S)+\NN^d}x^{w\cdot u}
 =\frac{x^{w\cdot\bigvee S}}{\prod_j(1-x^{w_j})}.
\]
Apply finite inclusion--exclusion to Lemma~\ref{lem:dickson}.
Positive weights ensure that every coefficient counts only finitely
many vectors. Thus these are valid formal-power-series identities,
with no convergence issue. For $U=\varnothing$ the numerator is zero.
\end{proof}

\section{The slab construction and the denominator}
Represent a multiset by its multiplicity vector
$m=(m_1,\ldots,m_k)\in\NN^k$. Its area is
$\sum_{j=1}^k jm_j$. Call $m$ feasible if this multiset tiles the
$2\times n$ rectangle where
$n=\tfrac12\sum_{j=1}^k jm_j$ is a nonnegative integer.
Let $\mathcal F_k$ be the set of feasible vectors.

If $m\in\mathcal F_k$, append a $2\times j$ slab tiled by two
horizontal $1\times j$ bars to the right of any witnessing tiling.
Writing $e_j$ for the $j$th coordinate vector gives
\begin{equation}\label{eq:slab}
 m\in\mathcal F_k\quad\Longrightarrow\quad
 m+2e_j\in\mathcal F_k\qquad(1\leq j\leq k).
\end{equation}
This is only an addition statement. Removing two bars from an arbitrary
feasible multiset is not assumed to preserve feasibility.

Every vector has a unique expression $m=\varepsilon+2u$, with
$\varepsilon\in\{0,1\}^k$ and $u\in\NN^k$. Only parity vectors in
\[
 E_k=\left\{\varepsilon\in\{0,1\}^k:
                  \sum_{j=1}^k j\varepsilon_j\text{ is even}\right\}
\]
can be feasible. For $\varepsilon\in E_k$, define
\[
 b_\varepsilon=\frac12\sum_{j=1}^k j\varepsilon_j,
 \qquad
 U_\varepsilon=\{u\in\NN^k:\varepsilon+2u\in\mathcal F_k\}.
\]
By \eqref{eq:slab}, every $U_\varepsilon$ is upward closed.
Applying Lemma~\ref{lem:series} with weights $w_j=j$ now gives
\begin{align*}
 \sum_{n\geq0}f_k(n)x^n
 &=\sum_{\varepsilon\in E_k}x^{b_\varepsilon}
                 \sum_{u\in U_\varepsilon}x^{\sum_jju_j}\\
 &=\frac{\displaystyle\sum_{\varepsilon\in E_k}
                 x^{b_\varepsilon}A_{U_\varepsilon}(x)}
              {\prod_{j=1}^k(1-x^j)}.
\end{align*}
There are finitely many parity vectors, and each numerator is a
polynomial with integer coefficients. This proves \eqref{eq:gf}.
Each multiset is counted once by its unique multiplicity vector;
neither distinct tilings nor different possible slab decompositions
introduce multiplicities.

\section{Period and exact degree}
Put $L=\operatorname{lcm}(1,\ldots,k)$. Since every $j$ divides $L$,
\eqref{eq:gf} can be rewritten as
\[
 \sum_{n\geq0} f_k(n)x^n
 =\frac{B_k(x)}{(1-x^L)^k},\qquad
 B_k(x)=P_k(x)\prod_{j=1}^k
                (1+x^j+\cdots+x^{L-j})\in\mathbb Z[x].
\]
Write $B_k(x)=\sum_{a=0}^{D}\beta_ax^a$. For $n\geq D$, coefficient
extraction using $(1-z)^{-k}=\sum_{h\geq0}\binom{h+k-1}{k-1}z^h$
gives
\begin{equation}\label{eq:quasi}
 f_k(n)=\sum_{\substack{0\leq a\leq D\\a\equiv n\pmod L}}
       \beta_a\binom{(n-a)/L+k-1}{k-1}.
\end{equation}
On each residue class this is a polynomial in $n$ of degree at most
$k-1$, with rational coefficients.

For completeness, elementary counting proves that the degree is
exactly $k-1$, without appealing to the source's asymptotics.
For $k=1$, only $2n$ unit squares are possible, so $f_1(n)=1$.
For $k\geq2$, set $W=\sum_{j=2}^k j$. Choose independently
\[
 0\leq a_j\leq\lfloor n/W\rfloor\quad(2\leq j\leq k),
 \qquad a_1=n-\sum_{j=2}^k ja_j\geq0.
\]
The multiset having $2a_j$ bars of length $j$ is feasible by giving
each row $a_j$ such bars. Distinct choices give distinct multisets.
Conversely, a vector of area $2n$ is determined by $m_2,\ldots,m_k$,
since $m_1=2n-\sum_{j=2}^k jm_j$, and each $m_j\leq2n$.
Therefore
\[
 (\lfloor n/W\rfloor+1)^{k-1}
 \leq f_k(n)\leq(2n+1)^{k-1}.
\]
The lower bound along every residue class in \eqref{eq:quasi} rules
out degree less than $k-1$. This proves the exact degree, and finishes
Theorem~\ref{thm:main}.

\section{Scope and verification}
The proof covers every fixed $k\geq1$ and establishes the entire
conjecture, including its precise denominator restriction. It does not
give optimal onset indices, minimal quasiperiods, or the explicit
numerators for arbitrary $k$. Dickson's lemma and upper-set generating
functions are classical; the application is the slab construction on
parity classes of feasible multisets. No claim of a new general
finite-generation theorem is made.

The companion \texttt{verify\_rectangle\_partitions.py} compares
literal cell tilings, deduplicated by their multiplicity vectors,
with an independent row-allocation test allowing vertical dominoes.
It also checks slab extensions and finite inclusion--exclusion
identities exactly. Its finite minimal-vector searches are labeled
bounded diagnostics, not certificates of all minimal vectors.
No finite computation is needed for the universal proof above.

As of September 5, 2026, the checked journal and preprint sources
state this conjecture. Searches by authors, title, conjecture number,
restricted tilings, quasipolynomials, and finite-generation mechanisms
located no prior resolution. The authors' August 2026 follow-ups
\cite{note,asymptotic} concern unrestricted rectangle asymptotics,
not this fixed-bar-length statement. This is a bounded literature
audit, not a guarantee against unindexed or unpublished work.

\end{document}